\pdfoutput=1
\def\smart{0} 

\documentclass[11pt,a4paper,oneside]{amsart}

\usepackage{ifthen}

\usepackage{color,amssymb,latexsym,amsfonts,textcomp,fancyhdr,calc,graphicx}
\usepackage{amsmath,amstext,amsthm,amssymb,amsxtra}

\ifthenelse{\smart=0}{
\setlength{\textwidth}{\paperwidth}
\addtolength{\textwidth}{-2in}
\calclayout
}{
\usepackage[paperwidth=10cm,paperheight=20cm,left=5mm,right=5mm,top=10mm,bottom=5mm]{geometry}
}
\usepackage[colorlinks,citecolor=red,pagebackref,hypertexnames=false]{hyperref}
\usepackage{dsfont}
\usepackage[framemethod=tikz]{mdframed}
\usepackage[inline]{asymptote}

\usepackage{amsfonts}
\usepackage{amsmath}
\usepackage{amssymb}
\usepackage{graphicx}
\usepackage{amscd}
\usepackage{xcolor}
\usepackage{hyperref}
\usepackage[normalem]{ulem}
\usepackage{bbm}
\usepackage{enumitem}

\usepackage{txfonts,pxfonts,tikz} 
\usepackage{tgschola}
\usepackage[T1]{fontenc}

\usepackage{color} \usepackage{soul}

\numberwithin{equation}{section}

\theoremstyle{plain}
\newtheorem{theorem}{Theorem}[section]
\newtheorem{lemma}{Lemma}[section]
\newtheorem{corollary}[theorem]{Corollary}

\newtheorem{question}{Question}

\theoremstyle{definition}
\newtheorem{definition}{Definition}[section]
\newtheorem{remark}{Remark}[section]

\newtheorem{case[theorem]}{Case}

\numberwithin{equation}{section}

\newcommand{\pr}{\mathbb{P}}

\newcommand{\beql}[1]{\begin{equation}\label{#1}}
\newcommand{\eeq}{\end{equation}}
\newcommand{\comment}[1]{}
\newcommand{\Ds}{\displaystyle}

\newcommand{\Abs}[1]{{\left|{#1}\right|}}

\newcommand{\Mean}{{\mathbb{E}}}

\newcommand{\Ceil}[1]{{\left\lceil{#1}\right\rceil}}

\newcommand{\Prob}[1]{{{\mathbb{P}}\left[{#1}\right]}}
\newcommand{\Set}[1]{{\left\{{#1}\right\}}}

\renewcommand{\AA}{{\mathbb A}}

\newcommand{\RR}{{\mathbb R}}

\newcommand{\ZZ}{{\mathbb Z}}
\newcommand{\NN}{{\mathbb N}}

\newcommand{\One}[1]{{\mathds{1}}\left(#1\right)}

\makeatletter
\@namedef{subjclassname@2020}{\textup{2020} Mathematics Subject Classification}
\makeatother

\begin{document}

\ifthenelse{\smart=1}{\sloppy}{}
	
\title[Large sets with no copies of a sequence]{Large sets containing no copies of a given infinite sequence}

\author{Mihail N. Kolountzakis}
\address{Department of Mathematics and Applied Mathematics, University of Crete, Voutes Campus, 70013 Heraklion, Crete, Greece.}
\email{kolount@gmail.com}

\author{Effie Papageorgiou}
\address{Department of Mathematics and Applied Mathematics, University of Crete, Voutes Campus, 70013 Heraklion, Crete, Greece.}
\email{papageoeffie@gmail.com}

\date{August 4, 2022. Revised August 19, 2023.}

\thanks{Supported by the Hellenic Foundation for Research and Innovation, Project HFRI-FM17-1733 and by University of Crete Grant 4725}

\begin{abstract}
Suppose $a_n$ is a real, nonnegative sequence that does not increase exponentially. For any $p<1$ we construct a Lebesgue measurable set $E \subseteq \RR$ which has measure at least $p$ in any unit interval and which contains no affine copy $\Set{x+ta_n:\ n\in\NN}$ of the given sequence (for any $x \in \RR, t > 0$). We generalize this to higher dimensions and also for some ``non-linear'' copies of the sequence. Our method is probabilistic.
\end{abstract}



\maketitle

\tableofcontents

\setlength{\parskip}{0.5em}

\section{Introduction}

In Euclidean Ramsey Theory one is interested in assuming some kind of largeness for sets $E$ in Euclidean space $\RR^d$, or, sometimes in $\ZZ^d$, and concluding that $E$ then contains a ``copy'' of a pattern. The most famous such example is perhaps Szemeredi's Theorem \cite{szemeredi1975sets} which states that any subset of the integers with positive density contains aribtrarily long arithmetic progressions. Another well known example is the theorem of Falconer and Marstrand \cite{falconer1986plane}, Furstenberg, Katznelson and Weiss \cite{furstenberg1990ergodic} and Bourgain \cite{bourgain1986szemeredi} (see also \cite{kolountzakis2004distance}) that if the set $E \subseteq \RR^d$ has positive Lebesgue density (this means that there are arbitrarily large cubes where $E$ takes up at least a constant fraction of the measure) then its points implement all sufficiently large distances (conjecture by Sz\'ekely \cite{szekely1983remarks}).

Another well known problem, very much related to the contents of this paper, is the so-called \textit{Erd\H os similarity problem}: A set $\AA \subseteq \RR$ is called {\em universal in measure} if whenever $E \subseteq \RR$ has positive Lebesgue measure we can find an affine copy of $A$ contained in $E$. In other words $x + t\AA \subseteq E$ for some $x \in \RR, t > 0$. It is easy to see that every finite set $\AA$ is universal (just look close enough to some point of density of $E$, shrink $\AA$ enough and average the number of points of the copy of $\AA$ that belong to $E$ over translates of $\AA$ nearby) but it has been conjectured \cite{erdos2015my} (see also \cite[p.\ 183]{croft2012unsolved}) that no infinite set $\AA$ can be universal in measure. This is known for many classes of infinite sets but not for all  \cite{falconer1984problem,gallagher2022topological,humke1998visit,komjath1983large,chlebik2015erd}. Clearly it would suffice to prove this for $\AA$ being a positive sequence $a_n$ decreasing to $0$ but if $a_n$ decays fast to 0 (so it is in some sense sparse, hence not that hard to contain) this is still unknown. On the contrary this is known when $\log \frac{1}{a_n} = o(n)$. This is not known if $a_n = 2^{-n}$, for example.

In this paper we consider an analogue of the Erd\H os similarity problem ``in the large''. Let $\AA \subseteq \RR$ be a discrete, unbounded, infinite set in $\RR$. Can we find a ``large'' measurable set $E \subseteq \RR$ which does not contain any affine copy $x+t\AA$ of $\AA$ (for any $x \in \RR, t > 0$)? Our attention to this problem was drawn by a recent paper by Bradford, Kohut and Mooroogen \cite{bradford2023large} in which the authors prove that if $\AA$ is an infinite arithmetic progression then this is indeed possible: for any $p \in [0, 1)$ they construct a Lebesgue measurable set $E$, with measure at least $p$ in any interval of length 1, which does not contain any affine copy of $\AA$. This is clearly equivalent to being able to obtain, for any $p \in [0, 1)$ a set $E$ avoiding all infinite arithmetic progressions and having measure $\ge p$ in any interval of length 1 \textit{whose endpoints are integers}.
(Indeed, if the set $E$ has measure at least $p$ in every interval of the form $[n, n+1]$, $n\in\ZZ$, then, since for any $x$ the interval $[x, x+1]$ is contained in the union of two such unit-length intervals with integer endpoints, we obtain that $[x, x+1]\setminus E$ has measure at most $2(1-p)$. Since $p$ can be as close to 1 as we want, this implies that $[x, x+1]\setminus E$ has measure as close to 0 as we want.)
From now on we follow this simplification and we deal only with intervals with integer endpoints (in any dimension).

We generalize the result of \cite{bradford2023large} to sequences of nonnegative numbers $\AA$ which do not grow too fast. To state our result, we introduce the following class of sequences.
	
\begin{definition}\label{def:class-A}
We say that a real sequence  $\AA = \Set{a_n,\ n \in \NN}$ is in the \textit{class (A)} if 
\begin{enumerate}[label=(\arabic*), ref=\ref{def:class-A}.(\arabic*)]
\item \label{zero} $a_0=0$,
\item  $a_{n+1}-a_{n} \geq 1$, for every $n\in \mathbb{N}$.
\item\label{growth} $\log a_n = o(n)$
\end{enumerate}
\end{definition}
\noindent\textbf{Remark}. Since the problem we are studying is translation invariant 
condition \ref{zero} in Definition \ref{def:class-A} is unnecessary, but we keep it as it
simplifies the proofs somewhat.

Writing
\beql{counting}
A(t) = \Abs{\AA \cap [0, t]}
\eeq
for the \textit{counting function} of the set $\AA$, notice that the growth condition \ref{growth} is equivalent to the limit, as $t \to +\infty$,
\beql{growth1}
\frac{A(t)}{\log t} \to +\infty.
\eeq

Our main result is the following.
\begin{theorem}\label{MainThm}
Consider the sequence $\mathbb{A}=\{a_n: \; n\in \mathbb{N}\}$ which belongs to the class (A). Then, for each $ 0\leq p<1$, there exists a Lebesgue measurable set $E\subseteq \mathbb{R}$ such that
$$
\Abs{E \cap [m, m+1]} \ge p,\ \ \ \text{ for all } m \in \ZZ,
$$
but $E$ does not contain any affine copy of $\mathbb{A}$.
\end{theorem}

As in the case of the Erd\H os similarity problem described above, the sparser the set $\AA$ is the easier it should be to be contained in large sets, so it is not surprising that we had to impose a growth condition (to belong to the class (A)). It remains an open question if a similar set $E$ can be constructed when $\AA$ grows exponentially or faster.
\begin{question}\label{q:open-1}
Is there a sequence $0 < a_n \to +\infty$ and a number $p \in [0, 1)$ such that one can find an affine copy of $\AA=\Set{a_n:\ n \in \NN}$ in any set $E \subseteq \RR$ which has measure more than $p$ in any interval of length 1?
\end{question}

Unlike the approach taken in \cite{bradford2023large} our method of proof is probabilistic. We construct a family of random sets and we show that, with high probability, such a random set will have all the properties we want. This method turns out to be extremely flexible, and this allows us to generalize. Not only can we deal with essentially arbitrary and unstructured sequences $\AA$ but we can also relax the sense in which we seek copies of $\AA$ in the large set $E$. Instead of scaling the elements of $\AA$ and translating them
$$
x+t a_n,\ \ x \in \RR, t > 0,
$$
we can allow for more general transformations
\beql{transformation}
x+\phi(n, t) \cdot a_n,\ \ x \in \RR, t > 0.
\eeq

\begin{theorem}\label{th:trans}
Consider the set $\mathbb{A}=\{a_n: \; n\in \mathbb{N}\}$, which belongs to the class (A), and let $\phi(n, t): \NN\times(0, +\infty) \to (0, +\infty)$ be such that for each $n$ the function $\phi(n, t)$ is increasing in $t$ and is such that for all $n \in \NN$ we have
\beql{phi-gap}
C_1 t \le \phi(n+1, t)a_{n+1}-\phi(n, t) a_n
\eeq
and
\beql{phi-upper}
\phi(n, t) \le C_2 t,\ \ \text{ for all } t>0,
\eeq
for some $C_1, C_2 > 0$.
Then, for each $0\leq p<1$, there exists a Lebesgue measurable set $E\subseteq \mathbb{R}$ such that $E$ intersects every interval of unit length in a set of measure at least $p$, but $E$ does not contain the set
$$
\Set{x+\phi(n, t) \cdot a_n: n \in \NN}
$$
for any choice of $x \in \RR, t>0$.
\end{theorem}

We adopt certain arguments from  \cite[Section 3]{kolountzakis1997infinite} where it is proved, on the Erd\H os similarity problem, that sequences with a finite limit, say $0$, which are not decaying very fast (e.g.\ they decay  polynomially or subexponentially but not, for instance, exponentially fast -- compare to our growth condition (\ref{growth})), cannot be universal in measure, by showing the existence of a randomly constructed set $E\subseteq [0,1]$, avoiding all affine copies of the sequence.

The measure assumption makes this problem different than other ``avoidance" problems, where the avoiding set is often taken to have zero Lebesgue measure but to have large Hausdorff
dimension or Fourier dimension. For example, in \cite{keleti2008construction},  a compact
subset of $\mathbb{R}$ is constructed that has full Hausdorff dimension but  does not contain any $3$-term arithmetic progression. See also \cite{cruz2022large,denson2021large,fraser2018large,maga2011full,mathe2017sets,shmerkin2017salem,yavicoli2021large}.

We can also prove the following result in higher dimension. We phrase it as avoiding linear images of a set in Euclidean space into another Euclidean space. In this manner we obtain easily some corollaries, Theorem \ref{MainThm} one of them, and its proof is rather simpler than that of Theorem \ref{MainThm} given in \S \ref{1d}. But it does not extend easily to more complicated transformations such as those in Theorem \ref{th:trans}, so we choose to stay with linear maps.
\begin{theorem}\label{th:infinite}
Let $d_1, d \ge 1$, $b, f >0$, $p \in [0, 1)$.
Let also $\alpha(R)$ be a function satisfying $\Ds\frac{\alpha(R)}{\log R} \to +\infty$ as $R \to +\infty$.

Then if $\AA \subseteq \RR^{d_1}$ is a discrete point set such that
\beql{growth2}
\Abs{\AA \cap B_R(0)} \le C_2 R^b,\ \ \ (R>0)
\eeq
there is a set $E \subseteq \RR^d$ such that
\begin{enumerate}[label=\roman*.]
\item $\Abs{E \cap (m+[0, 1]^d)} \ge p$ for all $m \in \ZZ^d$,
\item\label{inf-contain}
For any linear map $T:\RR^{d_1} \to \RR^d$ if for arbitrarily large values of $R$
\beql{inf-sep}
T(\AA) \cap B_R(0)
\eeq
contains at least $\alpha(R)$ points with separation $R^{-f}$ then
\beql{inf-containment}
T(\AA)  \text{ is not contained in } E.
\eeq
\end{enumerate}
\end{theorem}

\begin{proof}[Proof of Theorem \ref{MainThm} using Theorem \ref{th:infinite}]
Apply Theorem \ref{th:infinite} with $d_1 = 2, d = 1$, $b=1, \alpha(x)=A(x^{1/2})$ (where $A(x)$ is the counting function of $\AA$), $f = 1$ (there is great flexibility in choosing $\alpha(x), b, f$) and the set
$$
P = \AA \times \Set{1} \subseteq \RR^2
$$
to obtain a set $E\subseteq\RR$ satisfying $\Abs{E \cap [m, m+1]} \ge p$ for all $m \in \ZZ$. We see that \eqref{growth2} is satisfied. Let now $T:\RR^2 \to \RR$ be given by the $1 \times 2$ matrix $T = (t, x)$ so that
$$
T(P) = x+t\AA.
$$
For any $x \in \RR, t > 0$, the set $(x+t\AA) \cap [-R, R]$ contains at least $A(R/t)$ points of separation $t$, so, if $R$ is large enough,  it contains $\alpha(R) = A(R^{1/2})$ points with separation $R^{-1}$. It follows that $x+t\AA$ is not contained in $E$.
\end{proof}

\begin{corollary}[Avoiding linear images of general sets in high dimension]\label{cor:seq}
Let $p \in [0, 1)$, $d \ge 1$, $a_n \in \RR^d$, for $n \in \NN$, with $\log\Abs{a_n} = o(n)$ and $\Abs{a_n-a_{n+1}}\ge 1$ for all $n\in\NN$.
Then there is a set $E \subseteq \RR^d$ such that for all $m\in\ZZ^d$ we have $\Abs{E \cap (m+[0, 1]^d)} \ge p$ and such that for all $x \in \RR^d$ and for all non-singular linear $T:\RR^d\to\RR^d$ the set $\Set{x+Ta_n:\ n\in\NN}$ is not contained in $E$.
\end{corollary}

\begin{proof}
Take $\AA \subseteq \RR^{2d}$ to be the set $A \times \{(\underbrace{1, 0, \ldots, 0}_d)\}$, where $A = \Set{a_n:\ n \in \NN}$. Writing $A(s) = \#(A \cap B_s(0))$ for the counting function of $A$ we have $\Ds \frac{A(R)}{\log R} \to +\infty$. Use Theorem \ref{th:infinite} with $d_1=2d$, $b=1$, $\alpha(R) = A(R^{1/2})$, $f=1$. Let $T:\RR^d\to\RR^d$ be non-singular, $x \in \RR^d$, and define the linear map $S:\RR^{2d}\to\RR^d$ by
$$
S(u, v) = S(u, v_1, v_2, \ldots, v_d) = Tu + v_1 x.
$$
In other words the $d \times (2d)$ matrix of $S$ is $(T \ |\  x \ |\  0)$ in block form. It follows that
$$
S(\AA) = \Set{Ta_n + x:\ n \in \NN}.
$$
Since $T$ is non-singular it follows that if $R>0$ is sufficiently large the set $S(\AA) \cap B_R(0)$ contains at least $\alpha(R)$ points with separation $\ge R^{-1}$ so the set $E\subseteq\RR^d$ furnished by Theorem \ref{th:infinite} does not contain $S(\AA)$, as we had to prove. 
\end{proof}

\begin{corollary}[Corollary 6 from \cite{bradford2023large}]\label{cor:ap}
If $p \in [0, 1)$ then there exists a set $E \subseteq \RR^d$ such that $\Abs{E \cap (m+[0, 1]^d)} \ge p$ for all $m \in \ZZ^d$ and it does not contain any set of the form $x+\NN\Delta$, with $x \in \RR^d$, $\Delta \in \RR^d\setminus\Set{0}$ (an arithmetic progression in $\RR^d$).
\end{corollary}

\begin{proof}
We use Corollary \ref{cor:seq} with the sequence $a_n = (n, 0, \ldots, 0) \in \RR^d$, $x \in \RR^d$ and any non-singular $d\times d$ matrix $T$ that maps $(1, 0, \ldots, 0)$ to $\Delta$.
\end{proof}

The outline of this note is as follows. In \S \ref{1d} we give the proof of Theorem \ref{MainThm} without using Theorem \ref{th:infinite},  and we indicate how the same proof also works for Theorem \ref{th:trans}.
In \S \ref{2d} we extend our technique to cover linear transformations of given sequences from one Euclidean space to another and prove Theorem \ref{th:infinite} and some corollaries.

\noindent\textbf{Added in revision}: The results in \cite{burgin2022large}, which came after this paper was submitted, are very relevant to the results in this paper and contain some improvements.


\section{Warm-up and some basic tools: no translational copies} \label{warmup}

In this section we introduce the basic probabilistic method by proving the more restricted Theorem \ref{th:translates}: we can avoid all translations of a given infinite sequence $0 \le a_n \to +\infty$ with a set which is arbitrarily large everywhere. This is considerably easier than avoiding all affine copies of the sequence, when scaling the sequence as well as translating it is allowed. For translations we have only one degree of freedom while for affine copies we have two. Still, some important ingredients of the method will be evident in the proof of Theorem \ref{th:translates} below. In \S\ref{1d} we will introduce the extra discretization in scaling space that will be required.

\begin{theorem}\label{th:translates}
Let $\AA=\Set{a_0=0 < a_1 < a_2 < \cdots} \subseteq \RR$ be a sequence with $a_n \to +\infty$,  and $p \in [0, 1)$. Then we can find a Lebesgue measurable set $E \subseteq \RR$ such that no translate of $\AA$
$$
x+\AA, \ \ x \in \RR,
$$
is contained in $E$, and such that for each $m \in \ZZ$ we have
$$
\Abs{E \cap [m, m+1]} \ge p.
$$
\end{theorem}

\begin{proof}
Let $q<1$ be defined by $1-q=\frac12(1-p)$ (or $q = \frac12(1+p)$).
Passing to a subsequence we can assume that $a_{n+1}-a_n \ge 1$ for all $n$. We construct a random set $E$ by breaking up each unit interval $[m, m+1]$, $m\in\ZZ$, into a number $N_m$ of equal intervals and keeping each of these subintervals with probability $q$, independently, into our set $E$. As $\Abs{m}$ increases the number $N_m$ will also have to increase, so let us take $N_m = \max\Set{K, \Abs{m}}$ say, where the large positive integer $K$ will be determined later.

\begin{figure}[h]
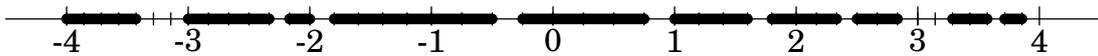

\begin{center}
\begin{asy}
unitsize(1.0cm);

int i, j, N=4;
int M[]={4, 5, 6, 7, 8};

srand(123);

draw((-N-0.5, 0) -- (N+0.5, 0));
for(i=-N; i<=N; ++i) {
 //dot((i, 0));
 draw((i,0.1)--(i,-0.05));
 label(string(i), (i,  0), S);
}
for(i=-N; i<N; ++i) {
 int m = i>=0?M[i]:M[-i-1];
 real h=1/m, d=0.05;

 for(j=0; j<m; ++j) {
  draw((i+j*h, d) -- (i+j*h, -d));
  if(unitrand()>0.3) { draw((i+j*h, 0)--(i+(j+1)*h, 0), linewidth(4bp)); }
 }
}
\end{asy}
\end{center}
\caption{How the random set $E$ looks like}
\end{figure}

Define now for $x \in \RR$ the random function
$$
\phi(x) = \One{x+\AA \subseteq E}.
$$
Since all points of $x+\AA$ are in different random intervals it follows, by independence, that $\Mean{\phi(x)} = \Prob{x + \AA \subseteq E} = 0$. Let the set of ``bad'' $x$ be
$$
B = \Set{x \in \RR:\ x+\AA \subseteq E}.
$$
We have
$$
\Mean{\Abs{B}} = \int \Mean{\phi(x)}\,dx = 0,
$$
hence $\Abs{B}$ is almost surely 0.

It remains to make sure that $\Abs{E \cap [m, m+1]} \ge p$ for all $m \in \ZZ$. Fix $m$ and let $X_1, \ldots, X_{N_m}$ be 0/1 random variables such that $X_i$ is 0 if we included the $i$-th subinterval of $[m, m+1]$ into the set $E$ and is 1 otherwise. In other words, $X_i$ denotes the absence of the $i$-the subinterval from the set $E$.
Clearly $\Mean{X_i} = 1-q$ and the random variable
$$
X = \sum_{i=1}^{N_m} X_i \text{ (the number of missing subintervals) }
$$
is a sum of independent indicator random variables with $\Mean{X} = (1-q) N_m$ and we can use the very versatile large deviation \textit{Chernoff inequality} (to be used repeatedly in \S\S \ref{Section3},\ref{2d} below)
\beql{chernoff-ineq}
\Prob{\Abs{X-\Mean{X}} \ge \epsilon \Mean{X}} \le 2 e^{-c_\epsilon \Mean{X}}
\eeq
(see \cite{chernoff1952measure,alon2016probabilistic}) with $\epsilon = 1$ to obtain
\begin{align}
\Prob{\Abs{E \cap [m, m+1]} < p} &=
  \Prob{X > (1-p)N_m} \nonumber\\
 & = \Prob{X-\Mean{X} > \Mean{X}}\nonumber\\
 & \le 2 \exp(-c_1  (1-q) \max\Set{K, \Abs{m}}).\label{chernoff-bound}
\end{align}
Define now the bad events $B_m = \Set{\Abs{E \cap [m, m+1]} < p}$ which we want not to hold, for all $m\in\ZZ$, and observe that the above inequality means that we can choose $K$ large enough to achieve
$$
\sum_{m \in \ZZ} \Prob{B_m} < \frac12.
$$
This means that with probability at least $1/2$ none of the bad events $B_m$ hold and, with the same probability, the set $B$ has measure 0.
We now amend our random set $E$ by removing from it the set $B$ (the set of first terms of those $x+\AA$ which are contained in $E$). Thus arises a set $E'$, which differs from $E$ by a set of measure $0$, and which contains no translate of $\AA$.
\end{proof}

\begin{remark}
It is not necessary to assume that $a_n \to +\infty$ in Theorem \ref{th:translates}. It suffices to assume that the set $\AA$ is infinite. If $\AA$ does not contain a sequence tending to infinity (for Theorem \ref{th:translates} to apply to it) then it will have a finite accumulation point, so a result of Komj\'ath \cite{komjath1983large} guarantees the existence of a set $\widetilde{E} \subseteq [0, 1]$, of measure arbitrarily close to $1$, which contains no translate of $\AA$.
Repeating $\widetilde{E}$ 1-periodically
$$
E = \bigcup_{n \in \ZZ} \widetilde{E}+n
$$
we obtain a set $E$ with the required properties. For a probabilistic proof of this result in the spirit of the present paper see \cite{kolountzakis1997infinite}.
\end{remark}

\begin{remark}
The Chernoff inequality \eqref{chernoff-ineq} is extremely useful when one needs to control a random variable $X$ (this means that one wants to ensure, with high probability, that $X$ is near its mean $\Mean{X}$) which is a sum of indicator, independent random variables. The key is that the mean $\Mean{X}$ cannot be very small, as it appears in the exponent in the right hand side of \eqref{chernoff-ineq}. Since one usually wants to do so \textit{simultanesouly} for a large number of random variables $X$, one key situation to keep in mind is the following: if the number of random variables to be controlled is polynomial in $N$ (a parameter) it is enough that their mean are at least a large multiple of $\log N$.
\end{remark}

With minor modifications of the proof we can get a progressively denser set $E$ avoiding all translates. We throw in the whole negative half line (as we could have done in Theorem \ref{MainThm} too).
\begin{theorem}\label{th:translates-denser}
Let $\AA=\Set{a_0=0 < a_1 < a_2 < \cdots} \subseteq \RR$ be a sequence with $a_n \to +\infty$. Then we can find a Lebesgue measurable set $E \subseteq \RR$ such that no translate of $\AA$
$$
x+\AA, \ \ x \in \RR,
$$
is contained in $E$, and such that
$$
(-\infty, 0] \subseteq E \ \text{ and }\  
\Abs{E \cap [m, m+1]} \to 1^{-} \text{ as } m \to +\infty.
$$
\end{theorem}

\begin{proof}
We indicate the differences with the proof of Theorem \ref{th:translates} and omit some details.

Our random set $E$ now will be of the same type as in the proof of Theorem \ref{th:translates} but with the probability of including the small subintervals tending slowly to $1$ as we go out to $+\infty$ and with the negative half line contained in $E$ to begin with.

Let us view the probability of keeping an interval as a function $p(s)$ defined on the real line. In the proof of Theorem \ref{th:translates} this function was constant. Here it will be constant on all intervals of the form $[m, m+1]$, $m \in \ZZ$.

With $\phi(x) = \One{x+\AA \subseteq E}$ we need again to ensure that $\Mean{\phi(x)}=0$ for all $x\in \RR$. After assuming, as in the previous proof, that the points of $\AA$ differ by at least 1, we again have independence of all events $x+a \in E$ for $a \in \AA$ so that $\Mean{\phi(x)}=0$ becomes equivalent to
$$
\prod_{a \in \AA} p(x+a) = 0,
$$
which, writing $q(s) = 1-p(s)$, is equivalent to
\beql{divergent}
\sum_{a \in \AA} q(x+a) = +\infty.
\eeq
Let $0 = k_1 < k_2 < \cdots$ be those positive integers for which
$$
[k, k+1) \cap \AA \neq \emptyset.
$$
Define then $q(x)$ to be $1/i$ in the interval $[{k_i}, {k_{i+1}})$, $i=1, 2, \ldots$.
It follows easily that for all $x \in \RR$ we have \eqref{divergent}: since the function $q(\cdot)$ is decreasing we have $q(x+a_n) \ge q(a_n)$ if $x \le 0$ and if $x \ge 0$ we have $q(x+a_n) \ge q(a_{\Ceil{x}+n})$ since $a_{k+1}-a_k \ge 1$ for all $k\in\NN$.
In both cases the series \eqref{divergent} contains a tail of the series $\sum_{a \in \AA} q(a)$ which is divergent.

It remains to ensure that the random variables $\Abs{[m, m+1]\setminus E}$ tend to 0 with ${m}\to+\infty$. These random variables are $\frac{1}{N_m}$ times a sum of independent indicator random variables (one for each of the $N_m$ subintervals into which we break up $[m, m+1]$)  of mean $q(m) N_m$ so we can use the Chernoff bound \eqref{chernoff-ineq} to obtain
$$
\Prob{\Abs{[m, m+1]\setminus E} > 2q(m)} \le 2\exp(-c_1 q(m) N_m).
$$
To ensure that the sum, over all $m\in\ZZ$ of the left hand side is $< 1$ we can of course pick the integers $N_m$ to be very large, say $N_m = K \frac{1}{q(m)} \Abs{m}$, with a sufficiently large constant $K>0$.

\end{proof}


\section{No affine copies for slowly increasing sequences} \label{1d}

In this section we prove Theorem \ref{MainThm} and explain why the proof also gives the more general Theorem \ref{th:trans}.

\begin{lemma}\label{Claim C}
Let $\mathbb{A}\in (A)$. For all $0<a<b$, $0\leq p<1$ and $\epsilon >0$, there is $N_0\in \mathbb{N}$, such that for all $N\geq N_0$, there is a set $ {E\subseteq [-N,N]}$ such that
\begin{enumerate}
\item[(i)]\label{(i)} for all $m\in \{-N,-N+1,..., N-1\}$, we have $\Abs{E\cap[m,m+1]}\geq p$, and
\item[(ii)]\label{(ii)} if the set $B$ consists of all $x\in [-N, N]$ for which  there is $t\in [a,b]$ such that 
\subitem (a) $(x+t\mathbb{A})\cap [-N, N] \subseteq E$ and 
\subitem (b) $\#((x+t\mathbb{A})\cap [-N, N])\geq A\left(\frac{N}{10b}\right)$, 
\end{enumerate}
then $\Abs{B}<\epsilon$. Here, $A(\cdot)$ is the counting function \eqref{counting} of the set $\AA$ and $A\left(\frac{N}{10b}\right) = \Abs{\AA \cap [0, N/(10b)]}$.
\end{lemma}

Let us first show how one derives Theorem \ref{MainThm}  from Lemma \ref{Claim C}. We give the proof of Theorem \ref{MainThm} in two steps: the first one verifies the result for a restricted scale, that is, for scales in a compact interval,  and the second one concludes for all positive scales, by writing the whole scaling interval  $(0,+\infty)$ as a countable union of intervals of the above type.

\textit{\textbf{Step 1}. For all $0<a<b$ and for each $0\leq p<1$,  there exists a set $E\subseteq \mathbb{R}$,   such that $\Abs{E\cap[m, m+1]}\geq p$ for all $m\in \mathbb{Z}$, but $E$ does not contain any affine copies of $\mathbb{A}$ with scale in $[a,b]$.}

Consider $ 0\leq p<1$ and a positive increasing sequence $\{p_n\}$, $n=1, 2, ...$ such that  $p_n \rightarrow 1^{-}$  and, moreover, 
\begin{equation}\label{measp}
	\sum\limits_{n=0}^{\infty}(1-p_n)<1-p.
\end{equation} 
Take also any positive sequence $\epsilon_n \rightarrow 0$. According to Lemma \ref{Claim C}, for $0<a<b$,  we can choose an increasing sequence of natural numbers $N_n = N_n(p_n, \epsilon_n, a, b) \to \infty$, for which there exist sets $E_n \subseteq[-N_n, N_n]$ with the following properties:  
\begin{itemize}
\item[(i)] for all $m=-N_n, ..., N_n-1$, we have $\Abs{E_n\cap[m,m+1]}\geq p_n$, 
\item[(ii)]  if 
\[\AA_n(x, t)=(x+t\mathbb{A})\cap[-N_n, N_n]\] 
and 
$$
B_n = \{x\in [-N_n, N_n]: \; \exists t\in [a,b] \text{ s.t. } \AA_n(x,t)\subseteq E_n
$$
$$
\text{ and } \# \AA_n(x,t)\geq A\left(\frac{N_n}{10b}\right)\},
$$
then $\Abs{B_n}<\epsilon_n$.
\end{itemize} 

Now take 
\[
\widetilde{E_n}=\left( -\infty, -N_n \right]\cup E_n \cup \left[N_n, +\infty \right)
\]
and
\[
E= \bigcap\limits_{n=1}^{\infty}\widetilde{E_n}.
\]

\begin{figure}[h]
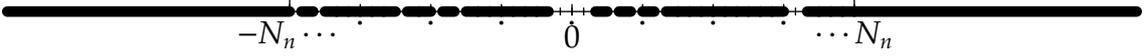

\begin{center}
\begin{asy}
size(9cm,0);
//unitsize(1.6cm);

int i, j, N=4;
//int M[]={4, 5, 6, 7, 8};
int m = 6;
real h=1/m, d=0.05;

srand(1234);

draw((-N-0.5, 0) -- (N+0.5, 0));
for(i=-N; i<=N; ++i) {
 //dot((i, 0));
 draw((i,0.1)--(i,-0.05));
 label(abs(i)<N?".":(i>0?"$\cdots N_n$":"$-N_n \cdots$"), (i,  0), S);
 if(abs(i)==N) draw((i, 0.2)--(i,-0.05));
}
label("$0$", (0, -d), S);

for(i=-N; i<N; ++i) {
 for(j=0; j<m; ++j) {
  draw((i+j*h, d) -- (i+j*h, -d));
  if(unitrand()>0.3) { draw((i+j*h, 0)--(i+(j+1)*h, 0), linewidth(4bp)); }
 }
}

draw((N, 0)--(2*N, 0), linewidth(4bp));
draw((-N, 0)--(-2*N, 0), linewidth(4bp));
\end{asy}
\end{center}
\caption{The set $\widetilde{E_n}$.}
\end{figure}

Then, since $\Abs{\widetilde{E_n}\cap[m, m+1]}\geq p_n$ for all $m\in \mathbb{Z}$, we get from (\ref{measp}) that the set $E$ has measure at least $p$ at every unit interval with integer endpoints. 
Also, if  there exist $x$, $t$ such that $x+t\mathbb{A}\subseteq E$, then $x+t\mathbb{A}$ is also contained in each $\widetilde{E_n}$. Having fixed $x$ and $t$ we can then find $n_0$ large enough such that for all $n\geq n_0$, we have $\# ((x+t\mathbb{A})\cap[-N_n,N_n])\geq A\left(\frac{N_n}{10b}\right)$. This implies that for every $n\geq n_0$, $x\in B_n$. It follows that for every $n\geq n_0$, $\Abs{B_n}<\epsilon_n$. Since $\epsilon_n \rightarrow 0$, setting
$$
B=\{x: \; \exists t\in [a,b] \text{ s.t. } x+t\mathbb{A} \subseteq E\},
$$
we get $\Abs{B}=0$. The null set of ``bad" translates $B$ is contained in $E$ (since we assumed that $0 \in \AA$), thus removing it from $E$ results in a set $E'$, which still has measure $\Abs{E'\cap[m, m+1]}\geq p$ for all $m\in \mathbb{Z}$, but contains no affine copy of $\mathbb{A}$ with scale in $[a,b]$.  

\textit{\textbf{Step 2}. Completion of the proof of Theorem \ref{MainThm}.}

Take a positive sequence $p_n' \in [0, 1)$,  {$n\in \mathbb{Z}$,} such that 
\begin{equation}\label{sum_pn}
	\sum\limits_{n\in \mathbb{Z}}(1-p_n')< {1-p}.
\end{equation}
Consider the intervals  $[a_n, b_n]=[2^{n-1}, 2^n]$,  $n\in \mathbb{Z}$. Then, according to Step 1, for each $p_n'$, there exists a set $E_n$ such that $\Abs{E_n\cap[m, m+1]}\geq p_n'$, for all $m\in \mathbb{Z}$, but for all $x\in \mathbb{R}$ and for all $t\in [a_n,b_n]$, the set $x+t\mathbb{A}$ is not contained in $E_n$.

Take
$$
E=\bigcap_{n\in \mathbb{Z}}E_n.
$$
Assume that for some $x\in \mathbb{R}$ and some $t>0$, $x+t\mathbb{A} \subseteq E$. Then, $x+t\mathbb{A} \subseteq E_n$, for all $n\in \mathbb{Z}$. However, since there is $n_0\in  {\mathbb{Z}}$ such that $t\in [2^{n_0-1}, 2^{n_0}]$, the inclusion $x+t\mathbb{A} \subseteq E_{n_0}$ cannot be true. Thus, $E$ does not contain any affine copy of $\mathbb{A}$ with positive scale. Finally, due to (\ref{sum_pn}) we have $\Abs{[m, m+1]\setminus E}<1-p$, or $\Abs{E\cap[m, m+1]} \geq p$. 		


\subsection{Proof of Lemma \ref{Claim C}} \label{Section3}

Fix the scale $t\in [a,b]$ and let $ 0\leq p<1$. 
Consider the positive sequence given by
\begin{equation}\label{pN}
	p_{N}=1-{\sqrt{\frac{\log\left( \frac{N}{10b}\right)}{A\left( \frac{N}{10b}\right)}}}.
\end{equation}
From \eqref{growth1} this implies 
$p_N\rightarrow 1^{-}$.

Partition $[-N,N]$ into unit intervals $[m, m+1]$, $m=-N, -N+1, ..., N-1$. Divide each $[m, m+1]$ further, into $k_N$ equal subintervals
\[
I_{i,m}=m+\left[ \frac{i-1}{k_N}, \frac{i}{k_N}\right],\quad i=1,..., k_{N}, 
\]
where  
\begin{equation}\label{k(N)'}
k_{N}=\left\lceil \frac{10}{a} \right\rceil \frac{N}{1-p_N}.
\end{equation}
Notice that $k_N/N \to +\infty$.

Construct a random set $E=E_N$ as follows: keep each $I_{i,m}$ in $E$ independently of the other intervals and with probability $p_N$ as in (\ref{pN}).
Then, $\pr(x\in E)=p_{N}$ for each $x\in [-N,N]$. 

\begin{figure}[h]
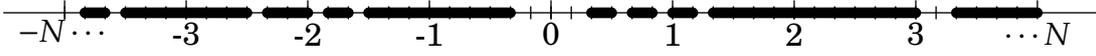

\begin{center}
\begin{asy}
unitsize(1.0cm);

int i, j, N=4;
//int M[]={4, 5, 6, 7, 8};

srand(1234);

draw((-N-0.5, 0) -- (N+0.5, 0));
for(i=-N; i<=N; ++i) {
 //dot((i, 0));
 draw((i,0.1)--(i,-0.05));
 label(abs(i)<N?string(i):(i>0?"$\cdots N$":"$-N \cdots$"), (i,  0), S);
}
for(i=-N; i<N; ++i) {
 int m = 6;
 real h=1/m, d=0.05;

 for(j=0; j<m; ++j) {
  draw((i+j*h, d) -- (i+j*h, -d));
  if(unitrand()>0.3) { draw((i+j*h, 0)--(i+(j+1)*h, 0), linewidth(4bp)); }
 }
}
\end{asy}
\end{center}
\caption{The random set $E$.}
\end{figure}

Let $M_{N}(x,t)$ be the number of elements of $(x+t\mathbb{A})\cap [-N,N]$ and observe that
\beql{mn}
M_N(x, t) \le A(2N/a),\ \ \text{ for } x \in [-N, N].
\eeq
For a given set $E\subseteq[-N, N]$, consider the set of ``bad" translates
$$
B=\Bigl\{ x \in [-N, N]:\; \exists t\in [a,b] \text{ s.t. }\\ (x+t\mathbb{A})\cap[-N,N]\subseteq E$$
$$
  \text{ and } M_N(x,t)\geq A\left(\frac{N}{10b}\right)\Bigr\}.
$$

We first deal with the measure of $B$.  We have
\begin{align}
	\mathbb{E}\Abs{B} &= \mathbb{E}\int_{-N}^{N}\mathbbm{1}_{B}(x)dx \notag\\
	&=\int_{-N}^{N}\mathbb{P}\Bigl[\exists t\in [a,b]: \; (x+t\mathbb{A})\cap[-N,N]\subseteq E\\
      &\text{ \ \ \ \ \ \ \ \ \ and } M_N(x,t)\geq A\left(  \frac{N}{10b} \right)\Bigr]\,dx.\label{measB'}
\end{align}
In what follows, we estimate from above the probability in (\ref{measB'}), uniformly in $x\in [-N, N]$.

Fix $x\in [-N,N]$.
To check whether there exists $t\in [a,b]$ such that $(x+t\mathbb{A})\cap[-N,N]\subseteq E$, it is sufficient to check whether such a $t$ exists in a finite set 
\beql{crossings}
S=S(x)=\{t_1, t_2, ..., t_u\}\subseteq [a,b].
\eeq
\begin{figure}[h]
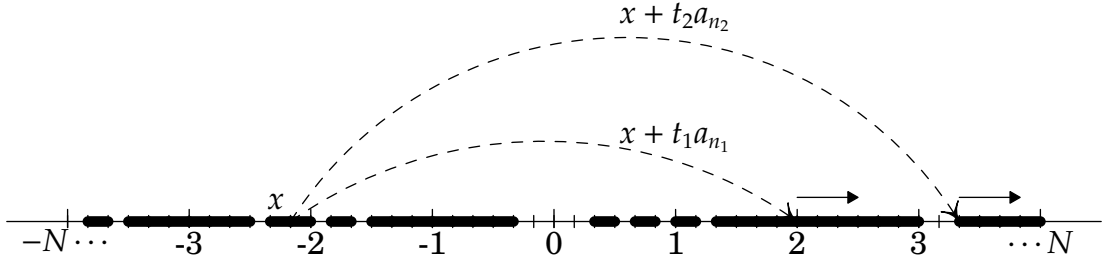

\begin{center}
\begin{asy}
unitsize(1.0cm);

int i, j, NN=4;
int m = 6;
real x, h=1/m, d=0.05;

srand(1234);

draw((-NN-0.5, 0) -- (NN+0.5, 0));
for(i=-NN; i<=NN; ++i) {
 //dot((i, 0));
 draw((i,0.1)--(i,-0.05));
 label(abs(i)<NN?string(i):(i>0?"$\cdots N$":"$-N \cdots$"), (i,  0), S);
}
for(i=-NN; i<NN; ++i) {
 for(j=0; j<m; ++j) {
  draw((i+j*h, d) -- (i+j*h, -d));
  if(unitrand()>0.3) { draw((i+j*h, 0)--(i+(j+1)*h, 0), linewidth(4bp)); }
 }
}

x = -2-h;
draw((x, -d)--(x, d));
label((x, d), "$x$", NW);

draw((x,0)..(1,0.5)..(2,0), dashed, ArcArrow(HookHead));
label((1, 0.5), "$x+t_1a_{n_1}$", N);
draw((2,4*d)--(2.5,4*d), ArcArrow);

draw((x,0)..(1,1.5)..(3+2*h,0), dashed, ArcArrow(HookHead));
label((1, 1.5), "$x+t_2a_{n_2}$", N);
draw((3+2*h,4*d)--(3.5+2*h,4*d), ArcArrow);
\end{asy}
\end{center}
\caption{As $x$ is held fixed and $t$ grows the points $x+t a_n$ cross over interval endpoints creating events that need to be checked.}
\end{figure}
Write $\alpha_0'<\alpha_1'<...<\alpha_{M_{N}(x, t)-1}'$ for the elements of $(x+t\mathbb{A})\cap[-N,N]$. Then, the set $S$ consists exactly of those $t\in [a,b]$ for which some $\alpha_j'=x+ta_j$, $j= 0, ..., M_{N}(x, t)-1$, is in the set $m+\left\{0, \frac{1}{k_N}, \frac{2}{k_N}, ..., \frac{k_N-1}{k_N}, 1 \right\}$, for some $m\in \{-N, -N+1, ..., N-1\}$.  Each of the points $\alpha_j'=x+ta_j$ traverses, as $t$ moves from $a$ to $b$, and as long as the point $\alpha_j'$ remains in $[-N, N]$, an interval of length at most $2N$, therefore it meets at most $2N k_N$ interval endpoints of the intervals $I_{i, m}$. Altogether, we have
\begin{equation}\label{numbert's2}
	u\leq 2N k_N \sup_{a \le t \le b} M_N(x, t) \leq c(a) N^2 (1-p_N)^{-1} A\left(  \frac{2N}{a} \right),
\end{equation}
where for the last inequality, we used \eqref{k(N)'} and \eqref{mn}.

Since $k_N \to +\infty$, we can take $N$ large enough, say $N\geq N_0$, so that $k_N>1/a$, for every $N\geq N_0$. Then, the length of each $I_{i,m}$ is small enough, $\le a$, to ensure that, for each $t\in [a,b]$, the points $\alpha_j'$, $j= 0, ..., M_{N}-1$, all belong to different intervals $I_{i,m}$. 
Therefore, for any fixed $x$ and $t$, 
\begin{align}
& \Prob{  (x+t\mathbb{A})\cap [-N, N]\subseteq E \text{ and } M_N(x,t)\geq A\left(  \frac{N}{10b} \right)} \nonumber\\
 \ \  &\ \ \le \Prob{  (x+t\mathbb{A})\cap [-N, N]\subseteq E \ \mid\   M_N(x,t)\geq A\left(  \frac{N}{10b} \right)} \nonumber\\
\ \  & \ \ \leq p_{N}^{A\left(  \frac{N}{10b} \right)}. \label{exponent}
\end{align}
Thus, using the bound (\ref{numbert's2}), 
\begin{align*}
&\Prob{\exists t\in S: \;  (x+t\mathbb{A})\cap [-N, N])\subseteq E}\\
&  \ \ \ \leq c(a)N^2 (1-p_N)^{-1}A\left(  \frac{2N}{a} \right)p_{N}^{A\left(  \frac{N}{10b} \right)}.
\end{align*}
Thus, (\ref{measB'}) yields
\[
\mathbb{E}\Abs{B}\leq 2 c(a)N^3 (1-p_N)^{-1}A\left(  \frac{2N}{a} \right)p_{N}^{A\left(  \frac{N}{10b} \right)}.
\]
We want to have  $$N^3(1-p_N)^{-1}A\left(  \frac{2N}{a} \right)p_{N}^{A\left(  \frac{N}{10b} \right)}\rightarrow 0,$$
while $p_{N}\rightarrow 1^{-}$, as $N\rightarrow \infty$. Since $A(\cdot)$ grows at most linearly at infinity, it suffices to show that
\begin{align}\label{limit}
	A\left(\frac{N}{10b}\right)&\log p_N \left( 4\frac{\log N}{A\left(\frac{N}{10b}\right)\log p_N}-\frac{\log(1-p_N)}{A\left(\frac{N}{10b}\right)\log p_N}+ 1 \right)\\
 &\rightarrow -\infty. \nonumber
\end{align}
To show (\ref{limit}), observe first that since $\lim_{x\rightarrow +\infty} x\log \left(1-x^{-1/2}\right)=-\infty,$
we have 	
\begin{equation}\label{limdom}
	\frac{A\left(\frac{N}{10b}\right)\log p_N}{\log N}\rightarrow -\infty,
\end{equation}
due to (\ref{pN}). Therefore, we also have $A\left(\frac{N}{10b}\right)\log p_N \rightarrow -\infty$. Finally, by \eqref{pN} and \eqref{limdom} we get 
$$\frac{\log(1-p_N)}{A\left(  \frac{N}{10b}\right)\log p_N }=-\frac{1}{2} \frac{\log A\left(\frac{N}{10b}\right)}{A\left(\frac{N}{10b}\right)\log p_N}\left\{ 1-\frac{\log \log\frac{N}{10b}}{\log A\left( \frac{N}{10b}\right)}\right\}\rightarrow 0.$$
In other words, we have shown that for every $\epsilon >0$, there is $N_1\geq N_0$ such that for all $N\geq N_1$, $\mathbb{E}\Abs{B}<\epsilon/2$, which implies that 
\begin{equation}\label{property1'}
	\pr(\Abs{B}\geq \epsilon)<1/2, \quad \forall N\geq N_1.
\end{equation}

We now turn to the measure of $E$ in every unit interval with integer endpoints. Fix $m\in [-N,N]$. Let $X_1^m, X_2^m, ..., X_{k_N}^m$ be independent indicator random variables, with $X_i^m=1$ if and only if $I_{i,m}\subseteq E$. Let $Y_i^m=1-X_i^m$ and denote by $X^m=\sum_{i=1}^{k_N}X_i^m$, $Y^m=\sum_{i=1}^{k_N}Y_i^m$ their sums. Then, $\mathbb{E}Y^m=(1-p_{N})k_{N}$. Notice also that the total measure kept in $[m,m+1]\cap E$ is equal to $X^{m}/k_N$.

For any $\delta>0$ we define the ``bad" events 
\[
A_m=\{ |Y^m-\mathbb{E}Y^m|>\delta\, \mathbb{E}Y^m  \}, \quad m= -N, -N+1, ..., N-1.
\]
To control $\Prob{A_m}$, we use Chernoff's inequality, \cite{alon2016probabilistic,chernoff1952measure}: for all $\delta >0$, 
\[
\Prob{A_m}\leq 2e^{-c_{\delta}\mathbb{E}Y^m}, 
\]
where $c_{\delta}=\min\Set{(1+\delta)\log(1+\delta)-\delta\log\delta,\, \delta^{2}/2}$.
Take $\delta=1/2$. It follows that 
\begin{align*}
& \Prob{|Y^m-(1-p_{N})k_{N}|>\frac{1}{2}(1-p_{N})k_{N}}\\
&\ \ \ \ \ \leq 2\exp\left(-{\frac{1}{2}(1-p_{N})k_{N}}\right).
\end{align*}

Thus, the probability that there is some $[m, m+1]\subseteq [-N,N]$, such that $A_m$ holds, is at most 
\[
4N\exp\left(-{\frac{1}{2}(1-p_{N})k_{N}}\right)
\]
and the right hand side tends to zero as $N\rightarrow +\infty$, by our choice of $k_{N}$ in (\ref{k(N)'}).
Thus, there is $N_2\geq N_1$ such that  
\begin{equation}\label{property2'}
	\Prob{\exists m\in \{-N, -N+1, ..., N-1\}: \, A_m \text{ holds}} < \frac12, 
\end{equation}
for all $N\geq N_2$.
Then, (\ref{property1'}) and (\ref{property2'}) imply the existence of a set $E\subseteq \mathbb{R}$ such that, on the one hand, it satisfies 
\[
\Abs{B}<\epsilon
\]
and on the other hand, 
\[
X^m-p_{N}k_{N}\geq -\frac{1}{2}(1-p_{N})k_{N}, 
\]
for all $m=-N, -N+1, ..., N-1$, for all $N\geq N_2$. Thus the measure of $E$ in each unit interval $[m, m+1]$,  is at least $p_N-\frac{1}{2}(1-p_N) \rightarrow 1$, as $p_N\rightarrow 1^{-}$. In other words, for all $0\leq p<1$, there is $N_3\geq N_2$ such that for all $N\geq N_3$, we have $\Abs{E\cap[m, m+1]}\geq p$. The proof of Lemma \ref{Claim C} is now complete.

\begin{remark}\label{rem:trans}
Let us indicate here why the proof of Theorem \ref{MainThm} just completed also applies to Theorem \ref{th:trans} without any essential changes. First of all, the implication from Lemma \ref{Claim C} to Theorem \ref{th:infinite} (finite to infinite) remains true almost verbatim. So it suffices to ensure that Lemma \ref{Claim C} is true in this case. The main ingredients of the proof of Lemma \ref{Claim C} are the following. Having fixed $x$ and varying $t$ we have to make sure that the following conditions hold.
\begin{enumerate}[label=C.\arabic*]
\item All points of the the $(x, t)$-copy of the set remain well separated, so that independence applies and we can multiply the probabilities that they belong to our random set. This is ensured by  \eqref{phi-gap}.
\item The number of points in the $(x, t)$-copy of the set in the interval $[-N, N]$ has to be large as this is the exponent in the upper bound \eqref{exponent}. Condition \eqref{phi-upper} guarantees this.
\item The number of events that need to be checked so that we are certain that for all $t$ no $(x, t)$-copy is contained in our random set is small. This is the number $u$ in \eqref{crossings}. What we are doing in the proof is to count how many times each of the points of our set (as $x$ is held fixed and $t$ increases from $a$ to $b$) crosses over an interval boundary. Since the $\phi(n, t)$ are assumed increasing in $t$ this remains as before.
\end{enumerate}
It should be clear that the conditions imposed on the scaling functions $\phi(n, t)$ in Theorem \ref{th:trans} are far from optimal. They are rather indicative of what can be accomplished with the method and it is clear that the method could work under different sorts of conditions.
\end{remark}

\section{The problem in higher dimension}\label{2d}

We will derive Theorem \ref{th:infinite} as a consequence of the more finitary theorem below.
\begin{theorem}\label{th:finite}
Let $d_1, d \ge 1$, $\beta, \zeta >0$, $p \in (0, 1)$.
Let also $\alpha(N)$ be a function satisfying $\Ds \frac{\alpha(N)}{\log N} \to +\infty$.

Then if $N$ is sufficiently large and $P \subseteq \RR^{d_1}$ is a point set with at most $N^\zeta$ points there is a set $E_N \subseteq [-N, N]^d$ such that
\begin{enumerate}
\item $\Abs{E_N \cap \left(m+[0, 1]^d\right)} \ge p$ for all $m = (m_1, \ldots, m_d) \in \ZZ^d$, with $-N \le m_j < N$,
\item\label{contain}
For any linear map $T:\RR^{d_1} \to \RR^d$ if
\beql{sep}
T(P) \cap [-N, N]^d
\eeq
contains at least $\alpha(N)$ points with separation$\ge N^{-\beta}$ then
\beql{containment}
\left( T(P) \cap [-N, N]^d \right) \subsetneq E_N.
\eeq
\end{enumerate}
\end{theorem}

\begin{proof}
Let $\gamma > \beta$ and split the cube $[-N, N]^d$ with a $N^{-\gamma}\times\cdots\times N^{-\gamma}$-spaced grid of $O(dN^{1+\gamma})$ hyperplanes perpendicular to the $d$ coordinate axes. Define the random set $E$ to contain each of the $N^{-\gamma}\times\cdots\times N^{-\gamma}$-sized cubes independently with probability $p' \in (p, 1)$. We show that with positive probability one can take $E_N = E$.

The first property of $E$ is a simple consequence of Chernoff bounds and we can assume it holds with probability $> \frac12$ working as in the proof of Theorem \ref{MainThm}.

Let $T = (T_{i, j})$ be a linear map $\RR^{d_1} \to \RR^d$. This depends on $d \cdot d_1$ real variables $T_{i, j}$, so we view $T$ as an element of $\RR^{d\cdot d_1}$. Instead of checking condition \eqref{contain} for all $T \in \RR^{d\cdot d_1}$ we first show that there is a small number (polynomial in $N$) of $T$'s that need to be checked.

Indeed, the set of $N^{-\gamma}\times\cdots\times N^{-\gamma}$-sized cubes that contain $T(P)$ does not change when $T$ varies except when one or more of the points in $T(P)$ cross a dividing hyperplane of those that subdivide $[-N, N]^d$. Let $H$ be one of those $O(d N^{1+\gamma})$ hyperplanes and fix an arbitrary point $h \in H$. Let also $u$ be a unit vector orthogonal to $H$. For a point $x \in \RR^d$ to belong to $H$ it must satisfy the linear equation
$$
E(H, x):\ \ u \cdot x = u \cdot h.
$$
Let $q \in P$. For the point $T(q)$ to belong to $H$ we must have
$$
E(H, T(q)):\ \ u \cdot T(q) = u \cdot h,
$$
which is a linear equation in $T \in \RR^{d\cdot d_1}$. Taking all such equations in $T$, over all dividing hyperplanes $H$ and all $q \in P$ we obtain a subdivision of $\RR^{d\cdot d_1}$ by
$$
n = O(d \cdot N^{1+\gamma} \cdot \Abs{P})
$$
hyperplanes. These $n$ hyperplanes subdivide $\RR^{d\cdot d_1}$ into $m = O(n^{d \cdot d_1})$ connected regions (this is easily proved by induction on the dimension, or see \cite{buck1943partition}). For any two points $T_1, T_2$ in the same region condition \eqref{containment} is either true for both or false for both since we can move continuously from $T_1$ to $T_2$ without leaving the region and, therefore, without any of the point $T(q)$ touching any of the dividing hyperplanes $H$.

\begin{figure}[h]
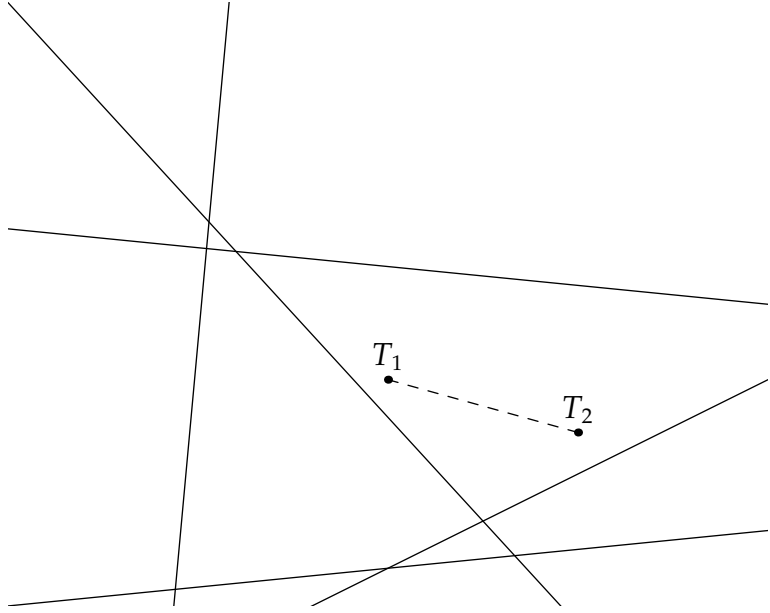

\begin{center}
\begin{asy}
size(9cm,0);
//unitsize(1.6cm);

draw( (0, 0) -- (10, 1) );
draw( (0, -2) -- (10, 3) );
draw( (0, 8) -- (10, -3) );
draw( (2, -2) -- (3, 9) );
draw( (0, 5) -- (10, 4) );
pair p1=(5,3), p2=(7.5, 2.3);
dot("$T_1$", p1, N);
dot("$T_2$", p2, N);
draw(p1--p2, dashed);
clip( box((0,0), (10, 8)) );

\end{asy}
\end{center}
\caption{The regions defined in $T$-space by the equations $E(H, T(q))$ for all $H, q$.
Only one of the transformations $T_1$, $T_2$ needs to be checked.}
\end{figure}

It suffices therefore to check condition \eqref{containment} for one point per region. Let us call these points $T_1, \ldots, T_m$. To guarantee that \eqref{containment} holds for all $T$ it is enough for it to be true for all $T_j$, $j=1, 2, \ldots, m$. Define the bad events
$$
B_j = \bigcap_{q \in P} \Set{T_j(q) \in E}.
$$
We need to ensure that none of the $B_j$ holds, but we only need to check those $B_j$
for which there is a $T$ in the cell of $T_j$ for which \eqref{sep} holds. For such a $j$ the number of different $N^{-\gamma}\times\cdots\times N^{-\gamma}$-sized cubes touched by $T_j(P)$ is the same as the number touched by $T(P)$ which is at least $\alpha(N)$ so
$$
\Prob{B_j} \le p'^{\alpha(N)},
$$
and it is therefore enough to make sure that
$$
n^{d \cdot d_1} p'^{\alpha(N)} = O\left(N^{\zeta \cdot d \cdot d_1} N^{(1+\gamma) d \cdot d_1} p'^{\alpha(N)}\right)
$$
can be made arbitrarily small by choosing $N$ large. This is clearly possible since the term $p'^{\alpha(N)}$ decays faster than any power of $N$.
\end{proof}

\begin{proof}[Proof of Theorem \ref{th:infinite}]
Let $p_n \in (0, 1)$ be such that
\beql{small-sum}
\sum_{n=1}^\infty (1-p_n) < 1-p.
\eeq
Apply Theorem \ref{th:finite} successively for $N=n$, $p_n$, $\zeta = b$, $\alpha(N)=\alpha(R)$, $\beta=f$ and the set $P = \AA \cap [-n, n]^{d_1}$ to obtain sets $E_n \subseteq [-n, n]^d$. Define
$$
E = \bigcap_{n=1}^\infty \left( E_n \cup (\RR^d \setminus [-n, n]^d)\right).
$$
It is easy to see because of \eqref{small-sum} that for any $m \in \ZZ^d$ we have $\Abs{E \cap m+[0, 1]^d} \ge p$. Let $T:\RR^{d_1}\to\RR^d$  and let $R$ be such that $T(\AA) \cap B_R(0)$ contains $\alpha(R)$ points which are $R^{-f}$ separated. Let $n=\Ceil{R}$. It follows from Theorem \ref{th:finite} that $T(\AA) \cap [-n, n]^d$ is not contained in $E_n \cup (\RR^d\setminus[-n, n]^d)$ and therefore not contained in $E$, as we had to show. 
\end{proof}

\bibliographystyle{alpha}
\bibliography{universal-sets}

\end{document}